\documentclass{amsart}
\usepackage{amsmath}
\usepackage{amssymb}
\usepackage{amsfonts}
\usepackage{graphicx}
\usepackage{amscd}

\newtheorem{theorem}{Theorem}
\theoremstyle{plain}

\newtheorem{definition}{Definition}

\newtheorem{proposition}{Proposition}
\newtheorem{remark}{Remark}

\numberwithin{equation}{section}

\input{tcilatex}

\begin{document}

\begin{center}
\bigskip {\Huge Soft Topology on Function Spaces}

\textbf{\ Taha Yasin \"{O}ZT\"{U}RK$^{\mathrm{a}}$ \textbf{and }Sadi BAYRAMOV%
\textbf{$^{\mathrm{a}}$}}

\medskip

$^{\mathrm{a}}$\textit{Department of Mathematics$,$ Faculty of Science and
Letters,}\\[0pt]
\textit{Kafkas University$,$ TR-}$36100$\textit{\ Kars$,$ Turkey}\\[0pt]

\textbf{e-mails: taha36100@hotmail.com, baysadi@gmail.com \ }

\bigskip

\textbf{Abstract}
\end{center}

\begin{quotation}
Molodtsov initiated the concept of soft sets in \cite{molod}. Maji et al.
defined some operations on soft sets in \cite{maji2}. The concept of soft
topological space was introduced by some authors. In this paper, we
introduce the concept of the pointwise topology of soft topological spaces
and the properties of soft mappings spaces. Finally, we investigate the
relationships between some soft mappings spaces.
\end{quotation}


\section{\protect\bigskip INTRODUCTION}

Many practical problems in economics, engineering, environment, social
science, medical science etc. cannot be dealt with by classical methods,
because classical methods have inherent difficulties. The reason for these
difficulties may be due to the inadequacy of the theories of
parameterization tools. Molodtsov \cite{molod} initiated the concept of soft
set theory as a new mathematical tool for dealing with uncertainties. Maji
et al. \cite{maji},\cite{maji2} research deal with operations over soft set.
The algebraic structure of set theories dealing with uncertainties is an
important problem. Many researchers heve contributed towards the algebraic
structure of soft set theory. Akta\c{s} and \c{C}a\u{g}man \cite{aktas}
defined soft groups and derived their basic properties. U. Acar et al. \cite%
{acar} introduced initial concepts of soft rings. F. Feng et al. \cite{feng}
defined soft semirings and several related notions to establish a connection
between soft sets and semirings. M. Shabir et al. \cite{shabir} studied soft
ideals over a semigrup. Qiu Mei Sun et al. \cite{sun} defined soft modules
and investigated their basic properties. C. Gunduz(Aras) and S. Bayramov 
\cite{gunduz1},\cite{gunduz} introduced fuzzy soft modules and
intuitionistic fuzzy soft modules and investigated some basic properties. T.
Y. Ozturk and S. Bayramov defined chain complexes of soft modules and their
soft homology modules. T. Y. Ozturk et al. introduced the concept of inverse
and direct systems in the category of soft modules.

Recently, Shabir and Naz \cite{shabir2} initiated the study of soft
topological spaces. Theoretical studies of soft topological spaces have also
been by some authors in \cite{shabir3},\cite{aygun},\cite{zorlutuna},\cite%
{min},\cite{cagman},\cite{hussain}. In the study \cite{bayramov} were given
different soft point concepts from the studies \cite{shabir3},\cite{aygun},%
\cite{zorlutuna},\cite{min},\cite{cagman},\cite{hussain}. In this study,
soft point consepts in the study \cite{bayramov} is used.

In the present study, the pointwise topology is defined in soft continuous
mappings space and the properties of soft mappings spaces is investigated.
Subsequently, the relations is given between on some soft mappings spaces.

\section{\textbf{PREL\.{I}M\.{I}NARY}}

In this section we will introduce necessary definitions and theorems for
soft sets. Molodtsov \cite{molod} defined the soft set in the following way.
Let $X$ be an initial universe set and $E$ be a set of parameters. Let $P(X)$
denotes the power set of $X$ and $A\subset E.$

\begin{definition}
\noindent \textbf{\ }\cite{molod} A pair $(F,A)$ is called a soft set over $%
X,$ where $F$ is a mapping given by $F:A\rightarrow P(X)$.

In other words, the soft set is a parameterized family of subsets of the set 
$X$. For $e\in A,$ $F(e)$ may be considered as the set of $\varepsilon -$%
elements of the soft set $(F,A),$ or as the set of $e-$approximate elements
of the soft set.
\end{definition}

\begin{definition}
\cite{maji2} For two soft sets $(F,A)$ and $(G,B)$ over $X$, $(F,A)$ is
called soft subset of $(G,B)$ if \noindent \textbf{\ }

\begin{enumerate}
\item $A\subset B$ and

\item $\forall e\in A$, $F(e)$ and $G(e)$ are identical approximations.

\noindent This relationship is denoted by $(F,A)\overset{\sim }{\subset }%
(G,B)$. Similarly, $(F,A)$ is called a soft superset of $(G,B)$ if $(G,B)$
is a soft subset of $(F,A)$. This relationship is denoted by $(F,A)\overset{%
\sim }{\supset }(G,B)$. Two soft sets $(F,A)$ and $(G,B)$ over $X$ are
called soft equal if $(F,A)$ is a soft subset of $(G,B)$, and $(G,B)$ is a
soft subset of $(F,A)$.
\end{enumerate}
\end{definition}

\begin{definition}
\noindent \cite{maji2} The intersection of two soft sets $(F,A)$ and $(G,B)$
over $X$ is the soft set $(H,C)$, where $C=A\cap B$ and $\forall e\in C$, $%
H(e)=F(e)\cap G(e)$. This is denoted by $(F,A)\overset{\sim }{\cap }%
(G,B)=(H,C)$.
\end{definition}

\begin{definition}
\noindent\ \textbf{\ \cite{maji2} }The union of two soft sets $(F,A)$ and $%
(G,B)$ over $X$ is the soft set, where $C=A\cup B$ and $\forall e\in C$, 
\begin{equation*}
H(\varepsilon )=\left\{ 
\begin{array}{l}
{F(e),\,\,\,\,\,\,\,\,\,\,\,\,\,\,\,\,\,\,\,\,\,\,\,\,\,\,\,\ \ \ \ \mathrm{%
if\;}}e\text{ }{\in A-B,} \\ 
{G(e),\,\,\,\,\,\,\,\,\,\,\,\,\,\,\,\,\,\,\,\,\,\,\,\,\,\,\ \ \ \,\mathrm{%
if\;}}e\text{ }{\in \mathrm{\;}B-A,} \\ 
{F(e)\cup G(e)\,\,\,\,\,\,\,\,\,\,\,\,\mathrm{if\;}}e\text{ }{\in A\cup B.}%
\end{array}%
\right.
\end{equation*}%
This relationship is denoted by $(F,A)\overset{\sim }{\cup }(G,B)=(H,C)$.
\end{definition}

\begin{definition}
\textbf{\cite{maji2} }A soft set $(F,A)$ over $X$ is said to be a NULL soft
set denoted by $\Phi $ if for all $e\in A,$ $F(e)=\varnothing $(null set).
\end{definition}

\begin{definition}
\textbf{\cite{maji2} }A soft set $(F,A)$ over $X$ is said to be an absolute
soft set denoted by $\overset{\sim }{X}$ if for all $e\in A,$ $F(e)=X$(null
set).
\end{definition}

\begin{definition}
\cite{shabir2} The difference $(H,E)$ of two soft sets $(F,E)$ and $(G,E)$
over $X$ , denoted by $(F,E)\backslash (G,E),$ is defined as $H(e)=F(e)/G(e)$
for all $e\in E.$
\end{definition}

\begin{definition}
\cite{shabir2} Let $Y$ be a non-empty subset of $X$ , then $\overset{\sim }{Y%
}$ denotes the soft set $(Y,E)$ over $X$ for which $Y(e)=Y,$ for all $e\in
E. $

In particular, $(X,E)$ will be denoted by $\overset{\sim }{X}.$
\end{definition}

\begin{definition}
\cite{shabir2} Let $(F,E)$ be a soft set over $X$ and $Y$ be a non-empty
subset of $X.$ Then the sub soft set of $(F,E)$ over $Y$ denoted by $\left(
^{Y}F,E\right) ,$ is defined as follows $^{Y}F(e)=Y\cap F(e),$ for all $e\in
E.$ In other words $\left( ^{Y}F,E\right) =\overset{\sim }{Y}\cap (F,E).$
\end{definition}

\begin{definition}
\cite{babitha} Let $(F,A)$ and $(G,B)$ be two soft sets over $X_{1}$ and $%
X_{2},$ respectively. The cartesian product $(F,A)\times (G,B)$ is defined
by $(F\times G)_{(A\times B)}$ where

\begin{equation*}
(F\times G)_{(A\times B)}(e,k)=\left( F,A\right) (e)\times (G,B)(k),\text{ \ 
}\forall (e,k)\in A\times B.
\end{equation*}

According to this definiton, the soft set $(F,A)\times (G,B)$ is soft set
over $X_{1}\times X_{2}$ and its parameter universe is $E_{1}\times E_{2}.$
\end{definition}

\begin{definition}
\cite{babitha} Let $(F_{1},E_{1})$ and $(F_{2},E_{2})$ be two soft sets over 
$X_{1}$ and $X_{2},$ respectively and $p_{i}:X_{1}\times X_{2}\rightarrow
X_{i},$ $q_{i}:E_{1}\times E_{2}\rightarrow E_{i}$ be projection mappings in
classical meaning. Then the soft mappings $\left( p_{i},q_{i}\right) ,$ $%
i\in \left\{ 1,2\right\} ,$ is called soft projection mapping from $%
X_{1}\times X_{2}$ to $X_{i}$ and defined by

\begin{eqnarray*}
\left( p_{i},q_{i}\right) \left( (F_{1},E_{1})\times (F_{2},E_{2})\right)
&=&\left( p_{i},q_{i}\right) \left( \left( F_{1}\times F_{2}\right) ,\left(
E_{1}\times E_{2}\right) \right) \\
&=&p_{i}\left( F_{1}\times F_{2}\right) ,q_{i}\left( E_{1}\times E_{2}\right)
\\
&=&(F,E)_{i}.
\end{eqnarray*}
\end{definition}

\begin{definition}
\cite{shabir2} Let $\tau $ be the collection of soft set over $X,$ then $%
\tau $ is said to be a soft topology on $X$ if

1) $\Phi ,\overset{\sim }{X}$ belongs to $\tau ;$

2) the union of any number of soft sets in $\tau $ belongs to $\tau ;$

3) the intersection of any two soft sets in $\tau $ belongs to $\tau .$

The triplet $\left( X,\tau ,E\right) $ is called a soft topological space
over $X.$
\end{definition}

\begin{definition}
\cite{shabir2} Let $\left( X,\tau ,E\right) $ be a soft topological space
over $X,$ then members of $\tau $ are said to be soft open sets in $X.$
\end{definition}

\begin{definition}
\cite{shabir2} Let $\left( X,\tau ,E\right) $ be a soft topological space
over $X.$ A soft set $(F,E)$ over $X$ is said to be a soft closed in $X,$ if
its relative complement $(F,E)^{\prime }$ belongs to $\tau $.
\end{definition}

\begin{proposition}
\cite{shabir2} Let $\left( X,\tau ,E\right) $ be a soft topological space
over $X.$ Then the collection $\tau _{e}=\{F(e):(F,E)\in \tau \}$ for each $%
e\in E,$ defines a topology on $X$.
\end{proposition}

\begin{definition}
\cite{shabir2} Let $\left( X,\tau ,E\right) $ be a soft topological space
over $X$ and $(F,E)$ be a soft set over $X.$ Then the soft closure of $%
(F,E), $ denoted by $\overline{(F,E)}$ is the isntersection of all soft
closed super sets of $(F,E).$ Clearly $\overline{(F,E)}$ is the smallest
soft closed set over $X$ which contains $(F,E).$
\end{definition}

\begin{definition}
\cite{shabir2} Let $x\in X,$ then $(x,E)$ denotes the soft set over $X$ for
which $x(e)=\{x\}$ for all $e\in E.$
\end{definition}

\begin{definition}
\cite{bayramov} Let $(F,E)$ be a soft set over $X$. The soft set $(F,E)$ is
called a soft point, denoted by $\left( x_{e},E\right) ,$ if for the element 
$e\in E,$ $F(e)=\{x\}$ and $F(e^{\prime })=\varnothing $ for all $e^{\prime
}\in E-\{e\}.$
\end{definition}

\begin{definition}
\cite{bayramov} Two soft points $\left( x_{e},E\right) $ and $\left(
y_{e^{\prime }},E\right) $ over a common universe $X$, we say that the
points are different points if $x\neq y$ or $e\neq e^{\prime }.$
\end{definition}

\begin{definition}
\cite{bayramov} Let $\left( X,\tau ,E\right) $ be a soft topological space
over $X.$ A soft set $(F,E)$ in $\left( X,\tau ,E\right) $ is called a soft
neighborhood of the soft point $\left( x_{e},E\right) \in (F,E)$ if there
exists a soft open set $(G,E)$ such that $\left( x_{e},E\right) \in
(G,E)\subset (F,E)$.
\end{definition}

\begin{definition}
\cite{gunduz2} Let $\left( X,\tau ,E\right) $ and $(Y,\tau ^{\prime },E)$ be
two soft topological spaces, $f:\left( X,\tau ,E\right) \rightarrow (Y,\tau
^{\prime },E)$ be a mapping. For each soft neighbourhood $(H,E)$ of $\left(
f(x)_{e},E\right) ,$if there exists a soft neighbourhood $(F,E)$ of $\left(
x_{e},E\right) $ such that $f\left( (F,E)\right) \subset (H,E),$ then $f$ is
said to be soft continuous mapping at $\left( x_{e},E\right) .$

If $f$ is soft continuous mapping for all $\left( x_{e},E\right) $, then $f$
is called soft continuous mapping.
\end{definition}

\begin{definition}
\cite{aygun} Let $\left( X,\tau ,E\right) $ be a soft topological space over 
$X.$ A subcollection $\beta $ of $\tau $ is said to be a base for $\tau $ if
every member of $\tau $ can be expressed as a union of members of $\beta .$
\end{definition}

\begin{definition}
\cite{aygun} Let $\left( X,\tau ,E\right) $ be a soft topological space over 
$X.$ A subcollection $\delta $ of $\tau $ is said to be a subbase for $\tau $
if the family of all finite intersections members of $\delta $ forms a base
for $\tau .$
\end{definition}

\begin{definition}
\cite{aygun} $\left\{ (\varphi _{i},\psi _{i}):(X,\tau ,E)\rightarrow \left(
Y_{i},\tau _{i},E_{i}\right) \right\} _{i\in \Delta }$ be a family of soft
mappings and $\left\{ \left( Y_{i},\tau _{i},E_{i}\right) \right\} _{i\in
\Delta }$ is a family of soft topological spaces. Then, the topology $\tau $
generated from the subbase $\delta =\left\{ (\varphi _{i},\psi _{i})_{i\in
\Delta }^{-1}(F,E):(F,E)\in \tau _{i},i\in \Delta \right\} $ is called the
soft topology (or initial soft topology) induced by the family of soft
mappings $\left\{ (\varphi _{i},\psi _{i})\right\} _{i\in \Delta }.$
\end{definition}

\begin{definition}
\cite{aygun} Let $\left\{ \left( X_{i},\tau _{i},E_{i}\right) \right\}
_{i\in \Delta }$ be a family of soft topological spaces. Then, the initial
soft topology on $X\left( =\tprod\nolimits_{i\in \Delta }X_{i}\right) $
generated by the family $\left\{ \left( p_{i},q_{i}\right) \right\} _{i\in
\Delta }$ is called product soft topology on $X.$ ( Here, $\left(
p_{i},q_{i}\right) $ is the soft projection mapping from $X$ to $X_{i},$ $%
i\in \Delta $).

The product soft topology is denoted by $\tprod\nolimits_{i\in \Delta }\tau
_{i}.$
\end{definition}

\section{Soft Topology on Function Spaces}

Let $\{(X_{s},\tau _{s},E)\}_{s\in S}$ be a family of soft topological
spaces over the same parameters set $E.$ We define a family of soft sets $%
\left( \underset{s\in S}{\dprod }X_{s},\tau ,E\right) $ as follows;

If $F_{s}:E\rightarrow P(X_{s})$ is a soft set over $X_{s}$ for each $s\in
S, $ then $\underset{s\in S}{\dprod }F_{s}:E\rightarrow P(\underset{s\in S}{%
\dprod }X_{s})$ is defined by $\left( \underset{s\in S}{\dprod }F_{s}\right)
(e)=\underset{s\in S}{\dprod }F_{s}(e).$ Let's consider the topological
product $\left( \underset{s\in S}{\dprod }X_{s},\underset{s\in S}{\dprod }%
\tau _{s},\underset{s\in S}{\dprod }E_{s}\right) $ of a family of soft
topological spaces $\{(X_{s},\tau _{s},E)\}_{s\in S}.$ We take the
contraction to the diagonal $\Delta \subset \underset{s\in S}{\dprod }E_{s}$
of each soft set $F:\underset{s\in S}{\dprod }E_{s}\rightarrow P(\underset{%
s\in S}{\dprod }X_{s}).$ Since there exist a bijection mapping between the
diagonal $\Delta $ and the parameters set $E,$ then the contractions of soft
sets are soft sets over $E$.

Now, let us define the topology on $\left( \underset{s\in S}{\dprod }%
X_{s},E\right) .$ Let $\left( p_{s_{0}},1_{E}\right) :\left( \underset{s\in S%
}{\dprod }X_{s},\tau ,E\right) \rightarrow (X_{s_{0}},\tau _{s_{0}},E)$ be a
projection mapping and the soft set $\left( p_{s_{0}},1_{E}\right)
^{-1}(F_{s_{0}},E)$ be for each $(F_{s_{0}},E)\in \tau _{s_{0}}$. Then

\begin{equation*}
\left( p_{s_{0}},1_{E}\right) ^{-1}(F_{s_{0}},E)=\left(
p_{s_{0}}^{-1}(F_{s_{0}}),E\right) =\left( F_{s_{0}}\times \underset{s\neq
s_{0}}{\dprod }\widetilde{X}_{s},E\right) .
\end{equation*}

The topology generated from $\left\{ \left. \left( F_{s_{0}}\times \underset{%
s\neq s_{0}}{\dprod }\widetilde{X}_{s},E\right) \right\vert s_{0}\in S,\text{
}(F_{s_{0}},E)\in \tau _{s_{0}}\right\} $ as a soft subbase and the soft
topology is denoted by $\tau =\underset{s\in S}{\dprod }\tau _{s}.$

\begin{definition}
The soft topological space $\left( \underset{s\in S}{\dprod }X_{s},\tau
,E\right) $ is called the product of family of the soft topological spaces $%
\{(X_{s},\tau _{s},E)\}_{s\in S}.$
\end{definition}

It is clear that each projection mapping $\left( p_{s},1_{E}\right) :\left( 
\underset{s\in S}{\dprod }X_{s},\tau ,E\right) \rightarrow (X_{s},\tau
_{s},E)$ is soft continuous. Additionally, the soft base of the soft
topology $\tau $ is formed by the soft sets

\begin{equation*}
\left\{ 
\begin{array}{c}
\left( F_{s_{1}}\times \underset{s\neq s_{1}}{\dprod }\widetilde{X}%
_{s},E\right) \cap ...\cap \left( F_{s_{n}}\times \underset{s\neq s_{n}}{%
\dprod }\widetilde{X}_{s},E\right)  \\ 
=\left( F_{s_{1}}\times ...\times F_{s_{n}}\times \underset{s\neq
s_{1}...s_{n}}{\dprod }\widetilde{X}_{s},E\right) 
\end{array}%
\right\} .
\end{equation*}

Let $\left( X,\tau ,E\right) $ be a soft topological space, $\{(Y_{s},\tau
_{s},E)\}_{s\in S}$ be a family of soft topological spaces and $\left\{
\left( f_{s},1_{E}\right) :\left( X,\tau ,E\right) \rightarrow (Y_{s},\tau
_{s},E)\right\} _{s\in S}$ be a family of soft mappings. For each soft point 
$x_{e}\in \left( X,\tau ,E\right) ,$ we define the soft map $f=\underset{%
s\in S}{\bigtriangleup }f_{s}:$ $\left( X,\tau ,E\right) \rightarrow \left( 
\underset{s\in S}{\dprod }Y_{s},\tau ,E\right) $ by $f(x_{e})=\left\{
f_{s}(x_{e})\right\} _{s\in S}=\left\{ (f_{s}(x))_{e}\right\} _{s\in S}.$ If 
$f:$ $\left( X,\tau ,E\right) \rightarrow \left( \underset{s\in S}{\dprod }%
Y_{s},\tau ,E\right) $ is any soft mapping, then $f=\underset{s\in S}{%
\bigtriangleup }f_{s}$ is satisfied for the family of soft mappings $\left\{
f_{s}=p_{s}\circ f:\left( X,\tau ,E\right) \rightarrow (Y_{s},\tau
_{s},E)\right\} _{s\in S}.$

\begin{theorem}
$f:\left( X,\tau ,E\right) \rightarrow \left( \underset{s\in S}{\dprod }%
Y_{s},\tau ,E\right) $ is soft continuous if and only if $f_{s}=p_{s}\circ
f:\left( X,\tau ,E\right) \rightarrow (Y_{s},\tau _{s},E)$ is soft
continuous for each $s\in S$.
\end{theorem}

\begin{proof}
$\Longrightarrow $ The proof is obvious.

$\Longleftarrow $ Let $\left( F_{s_{1}}\times ...\times F_{s_{n}}\times 
\underset{s\neq s_{1}...s_{n}}{\dprod }\widetilde{Y}_{s},E\right) $ be an
any soft base of product topology.

\begin{eqnarray*}
f^{-1}\left( F_{s_{1}}\times ...\times F_{s_{n}}\times \underset{s\neq
s_{1}...s_{n}}{\dprod }\widetilde{Y}_{s},E\right)  &=&f^{-1}\left(
p_{s_{1}}^{-1}(F_{s_{1}})\cap ...\cap p_{s_{n}}^{-1}(F_{s_{n}}),E\right)  \\
&=&\left( f^{-1}p_{s_{1}}^{-1}(F_{s_{1}})\cap ...\cap
f^{-1}p_{s_{n}}^{-1}(F_{s_{n}}),E\right) 
\end{eqnarray*}%
Since the soft mappings $p_{s_{1}}\circ f,...p_{s_{n}}\circ f$ are soft
continuous, the soft set 

\begin{equation*}
\left( f^{-1}p_{s_{1}}^{-1}(F_{s_{1}})\cap ...\cap
f^{-1}p_{s_{n}}^{-1}(F_{s_{n}}),E\right) 
\end{equation*}

is soft open. Thus, $f:\left( X,\tau ,E\right) \rightarrow \left( \underset{%
s\in S}{\dprod }Y_{s},\tau ,E\right) $ is soft continuous.
\end{proof}

If $\left\{ f_{s}:\left( X_{s},\tau _{s},E\right) \rightarrow (Y_{s},\tau
_{s}^{^{\prime }},E)\right\} _{s\in S}$ is a family of soft continuous
mappings, then the soft mapping $\underset{s\in S}{\dprod }f_{s}:\left( 
\underset{s\in S}{\dprod }X_{s},\tau ,E\right) \rightarrow \left( \underset{%
s\in S}{\dprod }Y_{s},\tau ^{^{\prime }},E\right) $ is soft continuous.

Now, let the family of soft topological spaces $\{(X_{s},\tau
_{s},E)\}_{s\in S}$ be discoint, i.e., $X_{s_{1}}\cap X_{s_{2}}=\varnothing $
for each $s_{1}\neq s_{2}$. For the soft set $F:E\rightarrow \underset{s\in S%
}{\cup }X_{s}$ over the set $E$, define the soft set $\left. F\right\vert
_{X_{s}}:E\rightarrow X_{s}$ by%
\begin{equation*}
\left. F\right\vert _{X_{s}}(e)=F(e)\cap X_{s},\text{ }\forall e\in E
\end{equation*}

and the soft topology $\tau $ define by

\begin{equation*}
(F,E)\in \tau \Longleftrightarrow \left( \left. F\right\vert
_{X_{s}},E\right) \in \tau _{s}.
\end{equation*}

It is clear that $\tau $ is a soft topology.

\begin{definition}
A soft topological space $\left( \underset{s\in S}{\cup }X_{s},\tau
,E\right) $ is called soft topological sum of the family of soft topological
spaces $\{(X_{s},\tau _{s},E)\}_{s\in S}$ and denoted by $\underset{s\in S}{%
\oplus }(X_{s},\tau _{s},E).$
\end{definition}

Let $\left( i_{s},1_{E}\right) :\left( X,\tau ,E\right) \rightarrow \underset%
{s\in S}{\oplus }(X_{s},\tau _{s},E)$ be an inclusion mapping for each $s\in
S.$ Since

\begin{equation*}
\left( i_{s},1_{E}\right) ^{-1}(F,E)=\left( \left. F\right\vert
_{X_{s}},E\right) \in \tau _{s},\text{ for }(F,E)\in \tau ,
\end{equation*}

$\left( i_{s},1_{E}\right) $ is soft continuous.

Let $\{(X_{s},\tau _{s},E)\}_{s\in S}$ be a family of soft topological
spaces, $(Y,\tau ^{^{\prime }},E)$ be a soft topological space and $\left\{
f_{s}:(X_{s},\tau _{s},E)\rightarrow (Y,\tau ^{^{\prime }},E)\right\} _{s\in
S}$ be a family of soft mappings. We define the function $f=\underset{s\in S}%
{\nabla }f_{s}:\underset{s\in S}{\oplus }(X_{s},\tau _{s},E)\rightarrow
(Y,\tau ^{^{\prime }},E)$ by $f(x_{e})=f_{s}(x_{e})=\left( f_{s}(x)\right)
_{e},$ where each soft point $x_{e}\in \underset{s\in S}{\oplus }(X_{s},\tau
_{s},E)$ can belong to a unique soft topological space $(X_{s_{0}},\tau
_{s_{0}},E).$ If $\ f:\underset{s\in S}{\oplus }(X_{s},\tau
_{s},E)\rightarrow (Y,\tau ^{^{\prime }},E)$ is an any soft mappings, then $%
\underset{s\in S}{\nabla }f_{s}=f$ is satisfied for the family of soft
mappings $\left\{ f_{s}=f\circ i_{s}:(X_{s},\tau _{s},E)\rightarrow (Y,\tau
^{^{\prime }},E)\right\} _{s\in S}$.

\begin{theorem}
The soft mapping $\ f:\underset{s\in S}{\oplus }(X_{s},\tau
_{s},E)\rightarrow (Y,\tau ^{^{\prime }},E)$ is a soft continuous if and
only if $f_{s}=f\circ i_{s}:(X_{s},\tau _{s},E)\rightarrow (Y,\tau
^{^{\prime }},E)$ are soft continuous for each $s\in S.$
\end{theorem}

\begin{proof}
$\Longrightarrow $ The proof is obvious.

$\Longleftarrow $ Let $(F,E)\in \tau ^{^{\prime }}$ be a soft open set. The
soft set $f^{-1}(F,E)$ belongs to the soft topology $\underset{s\in S}{%
\oplus }\tau _{s}$ if and only if the soft set $\left( \left.
f^{-1}(F)\right\vert _{X_{s}},E\right) $ belongs to $\tau _{s}.$Since

\begin{equation*}
\left( \left. f^{-1}(F)\right\vert _{X_{s}},E\right) =i_{s}^{-1}\left(
f^{-1}(F),E\right) =\left( i_{s}^{-1}\circ f^{-1}\right)
(F,E)=f_{s}^{-1}(F,E)\in \tau _{s},
\end{equation*}

$f$ is soft continuous.
\end{proof}

Let $\left\{ f_{s}:(X_{s},\tau _{s},E)\rightarrow (Y_{s},\tau _{s}^{^{\prime
}},E)\right\} _{s\in S}$ be a family of soft continuous mappings. We define
the mapping $f=\underset{s\in S}{\oplus }$ $f_{s}:\underset{s\in S}{\oplus }%
(X_{s},\tau _{s},E)\rightarrow \underset{s\in S}{\oplus }(Y_{s},\tau
_{s}^{^{\prime }},E)$ by $f(x_{e})=f_{s}(x_{e})$ where each soft point $%
x_{e}\in \underset{s\in S}{\oplus }(X_{s},\tau _{s},E)$ belongs to $%
(X_{s_{0}},\tau _{s_{0}},E).$ It is clear that if each $f_{s}$ is soft
continuous then $f$ is also soft continuous.

\begin{theorem}
Let $\{(X_{s},\tau _{s},E)\}_{s\in S}$ be a family of soft topological
spaces. Then

\begin{equation*}
\left( \underset{s\in S}{\dprod }X_{s},\tau _{e}\right) =\underset{s\in S}{%
\dprod }\left( X_{s},\tau _{s_{e}}\right) \text{ and }\left( \underset{s\in S%
}{\oplus }X_{s},\tau _{e}\right) =\underset{s\in S}{\oplus }\left(
X_{s},\tau _{s_{e}}\right)
\end{equation*}

are satisfied for each $e\in E.$
\end{theorem}

\begin{proof}
We should show that $\tau _{e}=\underset{s\in S}{\dprod }\left( \tau
_{s}\right) _{e}$. Let us take any set $U$ from $\tau _{e}.$ From the
definition of the topology $\tau _{e}$, there exist a soft open set 

\begin{equation*}
\left( (F_{s_{1}},E)\times ...\times (F_{s_{n}},E)\times \underset{s\neq
s_{1}...s_{n}}{\dprod }\widetilde{X}_{s}\right) 
\end{equation*}

such that the set $U=\left( F_{s_{1}}(e)\times ...\times F_{s_{n}}(e)\times 
\underset{s\neq s_{1}...s_{n}}{\dprod }\widetilde{X}_{s}\right) $ belongs to
the topology $\underset{s\in S}{\dprod }\left( \tau _{s}\right) _{e}.$

Conversely, let $\left( U_{s_{1}}\times ...\times U_{s_{n}}\times \underset{%
s\neq s_{1}...s_{n}}{\dprod }\widetilde{X}_{s}\right) \in \underset{s\in S}{%
\dprod }\left( \tau _{s}\right) _{e}$. Then from the definition of the
topology $\left( \tau _{s_{i}}\right) _{e}$, there exist soft open sets $%
(F_{s_{1}},E),...,(F_{s_{n}},E)$ such that $%
F_{s_{1}}(e)=U_{s_{1}},...,F_{s_{n}}(e)=U_{s_{n}}$. Then

\begin{equation*}
U_{s_{1}}\times ...\times U_{s_{n}}\times \underset{s\neq s_{1}...s_{n}}{%
\dprod }\widetilde{X}_{s}=F_{s_{1}}(e)\times ...\times F_{s_{n}}(e)\times 
\underset{s\neq s_{1}...s_{n}}{\dprod }\widetilde{X}_{s}\in \tau _{e}.
\end{equation*}

The topological sum can be proven in the sameway.
\end{proof}

Let $\left( X,\tau ,E\right) $ and $\left( Y,\tau ^{^{\prime }},E\right) $
be two soft topological spaces. $Y^{X}$ is denoted the all soft continuous
mappings from the soft topological space $\left( X,\tau ,E\right) $ to the
soft topological space $\left( Y,\tau ^{^{\prime }},E\right) $, i. e.,

\begin{equation*}
Y^{X}=\left\{ (f,1_{E}):\left( X,\tau ,E\right) \rightarrow \left( Y,\tau
^{^{\prime }},E\right) \mid (f,1_{E})-\text{a soft continuous map}\right\} .
\end{equation*}

If $(F,E)$ and $(G,E)$ are two soft set over $X$ and $Y$, respectively then
we define the soft set $\left( G^{F},E\right) $ over $Y^{X}$ as follows;

\begin{equation*}
G^{F}(e)=\left\{ (f,1_{E}):\left( X,\tau ,E\right) \rightarrow \left( Y,\tau
^{^{\prime }},E\right) \mid f\left( F(e)\right) \subset G(e)\right\} \text{
for each }e\in E.
\end{equation*}

Now, let $x_{\alpha }\in \left( X,\tau ,E\right) $ be an any soft point. We
define the soft mapping $e_{x_{\alpha }}:\left( Y^{X},E\right) \rightarrow
\left( Y,\tau ^{^{\prime }},E\right) $ by $e_{x_{\alpha }}(f)=f(x_{\alpha
})=\left( f(x)\right) _{\alpha }$. This mapping is called an evaluation map.
For the soft set $(G,E)$ over $Y,$ $e_{x_{\alpha }}^{-1}(G,E)=(G^{x_{\alpha
}},E)$ is satisfied.\ The soft topology that is generated from the soft sets 
$\left\{ (G^{x_{\alpha }},E)\mid (G,E)\in \tau ^{^{\prime }}\right\} $ as a
subbase is called pointwise soft topology and denoted by $\tau _{p}.$

\begin{definition}
$\left( Y^{X},\tau _{p},E\right) $ is called a pointwise soft function space
(briefly $PISFS$).
\end{definition}

\begin{remark}
The evaluation mapping $e_{x_{\alpha }}:\left( Y^{X},\tau _{p},E\right)
\rightarrow \left( Y,\tau ^{^{\prime }},E\right) $ is a soft continuous
mapping for each soft point $e_{x_{\alpha }}\in \left( X,\tau ,E\right) $.
\end{remark}

\begin{proposition}
A soft map $g:(Z,\eta ,E)\rightarrow \left( Y^{X},\tau _{p},E\right) $,
where $(Z,\eta ,E)$ is a soft topological space, is a soft continuous
mapping if and only if the soft mapping $e_{x_{\alpha }}\circ g:(Z,\eta
,E)\rightarrow \left( Y,\tau ^{^{\prime }},E\right) $ is a soft continuous
mapping for each $x_{\alpha }\in \left( X,\tau ,E\right) .$
\end{proposition}

\begin{theorem}
If the soft topological space $\left( Y,\tau ^{^{\prime }},E\right) $ is a
soft $T_{i}-$space for each $i=0,1,2$ then the soft space $\left( Y^{X},\tau
_{p},E\right) $ is also a soft $T_{i}-$space.
\end{theorem}

\begin{proof}
The soft points of the soft topological space $\left( Y^{X},\tau
_{p},E\right) $ denoted by $\left( f_{\alpha },E\right) $ i. e., if $\beta
\neq \alpha $ then $f_{\alpha }(\beta )=\varnothing $ and if $\beta =\alpha $
then $f_{\alpha }(\beta )=f.$ Now, let $f_{\alpha }\neq g_{\beta }$ be two
soft points. Then it should be $f\neq g$ or $\alpha \neq \beta $. If $f=g,$
then $\left( f(x)\right) _{\alpha }\neq \left( g(x)\right) _{\beta }\in
\left( Y,\tau ^{^{\prime }},E\right) $ for each $x\in X.$ If $f\neq g,$ then 
$f(x^{0})\neq g(x^{0})$ suct that $x^{0}\in X.$ Therefore, $\left(
f(x^{0})\right) _{\alpha }\neq \left( g(x^{0})\right) _{\beta }$ is
satisfied. In both cases, $\left( f(x^{0})\right) _{\alpha }\neq \left(
g(x^{0})\right) _{\beta }\in \left( Y,\tau ^{^{\prime }},E\right) $ is
satisfied for at least one $x_{0}\in X.$ Since $\left( Y,\tau ^{^{\prime
}},E\right) $ is a soft $T_{i}-$space, there exists soft open sets $%
(F_{1},E),(F_{2},E)\in \tau ^{^{\prime }}$ where the condition of the soft $%
T_{i}-$space is satisfied. Then the soft open sets $\left( F_{1}^{x_{\alpha
}^{0}},E\right) =e_{x_{\alpha }^{0}}^{-1}(F_{1},E)$ and $\left(
F_{2}^{x_{\beta }^{0}},E\right) =e_{x_{\beta }^{0}}^{-1}(F_{2},E)$ are
neighbours of soft points $f_{\alpha }$ and $g_{\beta },$ respectivelly
where the conditions of soft $T_{i}-$space are satisfied for these
neighbours.
\end{proof}

Now, we construct relationships betwen some function spaces. Let $%
\{(X_{s},\tau _{s},E)\}_{s\in S}$ be a family of pairwise disjoint soft
topological spaces, $\left( Y,\tau ^{^{\prime }},E\right) $ be a soft
topological spaces and $\underset{s\in S}{\dprod }\left( Y^{X_{s}},\tau
_{s_{p}},E\right) ,$ $\underset{s\in S}{\oplus }(X_{s},\tau _{s},E)$ be a
product and sum of soft topological spaces, respectively. Define

\begin{equation*}
\nabla :\underset{s\in S}{\dprod }\left( Y^{X_{s}},\tau _{s_{p}},E\right)
\rightarrow \left( Y^{\underset{s\in S}{\oplus }X_{s}},\left( \tau
^{^{\prime ^{\underset{s\in S}{\oplus }}}}\right) _{p},E\right)
\end{equation*}

such that $\forall \left\{ \left( f_{s},1_{E}\right) \right\} \in \underset{%
s\in S}{\dprod }Y^{X_{s}},$ $\forall x_{\alpha }\in \underset{s\in S}{\oplus 
}(X_{s},\tau _{s},E),$ $\underset{s\in S}{\nabla }\left( \left\{
f_{s}\right\} \right) \left( x_{\alpha }\right) =f_{s_{0}}\left( x_{\alpha
}\right) =\left( f_{s_{0}}\left( x\right) \right) _{\alpha }$ , where $%
x_{\alpha }$ belongs to unique $(X_{s_{0}},\tau _{s_{0}},E).$ We define the
mapping

\begin{equation*}
\nabla ^{-1}:\left( Y^{\underset{s\in S}{\oplus }X_{s}},\left( \tau
^{^{\prime ^{\underset{s\in S}{\oplus }}}}\right) _{p},E\right) \rightarrow 
\underset{s\in S}{\dprod }\left( Y^{X_{s}},\tau _{s_{p}},E\right)
\end{equation*}

by $\nabla ^{-1}(f)=\left\{ f\circ i_{s}=\left. f\right\vert
_{X_{s}}:X_{s}\rightarrow Y\right\} \in \underset{s\in S}{\dprod }\left(
Y^{X_{s}},\tau _{s_{p}},E\right) $ for each $f:\underset{s\in S}{\oplus }%
X_{s}\rightarrow Y$. It is clear that the mapping $\nabla ^{-1}$ is an
inverse of the mapping $\nabla .$

\begin{theorem}
The mapping

\begin{equation*}
\nabla :\underset{s\in S}{\dprod }\left( Y^{X_{s}},\tau _{s_{p}},E\right)
\rightarrow \left( Y^{\underset{s\in S}{\oplus }X_{s}},\left( \tau
^{^{\prime ^{\underset{s\in S}{\oplus }}}}\right) _{p},E\right)
\end{equation*}

is a soft homeomorphism in the pointwise soft topology.
\end{theorem}

\begin{proof}
To prove the theorem, it is sufficient to show that the mappings $\nabla $
and $\nabla ^{-1}$ are soft continuous. For this, we need to show that the
soft set $\nabla ^{-1}\left( e_{x_{\alpha }}^{-1}(F,E)\right) $ is a soft
open set, where each $e_{x_{\alpha }}^{-1}(F,E)$ belongs to a soft subbase
of the soft space $\left( Y^{\underset{s\in S}{\oplus }X_{s}},\left( \tau
^{^{\prime ^{\underset{s\in S}{\oplus }}}}\right) _{p},E\right) .$

\begin{equation*}
e_{x_{\alpha }}^{-1}(F,E)=\left\{ f:\underset{s\in S}{\oplus }%
X_{s}\rightarrow Y\mid f\left( x_{\alpha }\right) \in (F,E)\right\} =\left\{
f_{s_{0}}:X_{s_{0}}\rightarrow Y\mid f_{s_{0}}\left( x_{\alpha }\right) \in
(F,E)\right\} .
\end{equation*}

Since

\begin{eqnarray*}
\nabla ^{-1}\left( e_{x_{\alpha }}^{-1}(F,E)\right) &=&\nabla ^{-1}\left(
\left\{ f_{s_{0}}:X_{s_{0}}\rightarrow Y\mid f_{s_{0}}\left( x_{\alpha
}\right) \in (F,E)\right\} \right) \\
&=&\left\{ f_{s_{0}}:X_{s_{0}}\rightarrow Y\mid f_{s_{0}}\left( x_{\alpha
}\right) \in (F,E)\right\} \times \underset{s\neq s_{0}}{\dprod }\left(
Y^{X_{s}},\tau _{s_{p}},E\right)
\end{eqnarray*}

is the last soft set, $\nabla ^{-1}\left( e_{x_{\alpha }}^{-1}(F,E)\right) $
is a soft open set on the product space $\underset{s\in S}{\dprod }\left(
Y^{X_{s}},\tau _{s_{p}},E\right) .$

Now, we prove that the mapping $\nabla ^{-1}:\left( Y^{\underset{s\in S}{%
\oplus }X_{s}},\left( \tau ^{^{\prime ^{\underset{s\in S}{\oplus }}}}\right)
_{p},E\right) \rightarrow \underset{s\in S}{\dprod }\left( Y^{X_{s}},\tau
_{s_{p}},E\right) $ is soft continuous. Indeed, for each the soft set $%
\left( e_{x_{\alpha }}^{-1}(F,E)\right) _{s_{0}}(F,E)\times \underset{s\neq
s_{0}}{\dprod }\left( Y^{X_{s}},\tau _{s_{p}},E\right) $ belongs to the
subbase of the product space $\underset{s\in S}{\dprod }\left(
Y^{X_{s}},\tau _{s_{p}},E\right) ,$

\begin{equation*}
\left( e_{x_{\alpha }}^{-1}\right) _{s_{0}}(F,E)\times \underset{s\neq s_{0}}%
{\dprod }\left( Y^{X_{s}},\tau _{s_{p}},E\right) =\left\{ \left\{
f_{s}\right\} \in \underset{s\in S}{\dprod }Y^{X_{s}}\mid f_{s_{0}}\left(
x_{\alpha }\right) \in (F,E)\right\}
\end{equation*}

is satisfied.

Since the set

\begin{eqnarray*}
&&\left( \nabla ^{-1}\right) ^{-1}\left( \left( e_{x_{\alpha }}^{-1}\right)
_{s_{0}}(F,E)\times \underset{s\neq s_{0}}{\dprod }\left( Y^{X_{s}},\tau
_{s_{p}},E\right) \right)  \\
&=&\nabla \left( \left( e_{x_{\alpha }}^{-1}\right) _{s_{0}}(F,E)\times 
\underset{s\neq s_{0}}{\dprod }\left( Y^{X_{s}},\tau _{s_{p}},E\right)
\right)  \\
&=&\left\{ \underset{s\in S}{\nabla }f_{s}:f_{s_{0}}\left( x_{\alpha
}\right) \in (F,E)\right\} 
\end{eqnarray*}

belongs to subbase of the soft topological space $\left( Y^{\underset{s\in S}%
{\oplus }X_{s}},\left( \tau ^{^{\prime ^{\underset{s\in S}{\oplus }%
}}}\right) _{p},E\right) ,$ the mapping $\nabla ^{-1}$ is soft continuous.
Thus, the mapping

\begin{equation*}
\nabla :\underset{s\in S}{\dprod }\left( Y^{X_{s}},\tau _{s_{p}},E\right)
\rightarrow \left( Y^{\underset{s\in S}{\oplus }X_{s}},\left( \tau
^{^{\prime ^{\underset{s\in S}{\oplus }}}}\right) _{p},E\right)
\end{equation*}

is a soft homeomorphism.
\end{proof}

Now, $\{(Y_{s},\tau _{s}^{^{\prime }},E)\}_{s\in S}$ be the family of soft
topological spaces, $\left( X,\tau ,E\right) $ be a soft topological space.
We define mapping

\begin{equation*}
\Delta :\underset{s\in S}{\dprod }\left( Y^{X_{s}},\tau _{s_{p}}^{^{\prime
}},E\right) \rightarrow \left( \left( \underset{s\in S}{\dprod }Y_{s}\right)
^{X},\left( \underset{s\in S}{\dprod }\tau _{s}^{^{\prime }}\right)
_{p},E\right)
\end{equation*}

by the rule $\forall \left\{ f_{s}:X\rightarrow Y_{s}\right\} \in \underset{%
s\in S}{\dprod }\left( Y^{X_{s}},\tau _{s_{p}}^{^{\prime }},E\right) ,$ $%
\Delta \left\{ f_{s}\right\} =\underset{s\in S}{\Delta }f_{s}$.

Let the mapping $\Delta ^{-1}=\left( \left( \underset{s\in S}{\dprod }%
Y_{s}\right) ^{X},\left( \underset{s\in S}{\dprod }\tau _{s}^{^{\prime
}}\right) _{p},E\right) \rightarrow \underset{s\in S}{\dprod }\left(
Y^{X_{s}},\tau _{s_{p}}^{^{\prime }},E\right) $ be

\begin{equation*}
\Delta ^{-1}(f)=\left\{ p_{s}\circ f=f_{s}:X_{s}\rightarrow Y\right\}
\end{equation*}

for each $f\in \left( \left( \underset{s\in S}{\dprod }Y_{s}\right)
^{X},\left( \underset{s\in S}{\dprod }\tau _{s}^{^{\prime }}\right)
_{p},E\right) .$ It is clear that the mapping $\Delta ^{-1}$ is an inverse
of the mapping $\Delta .$

\begin{theorem}
The mapping 
\begin{equation*}
\Delta :\underset{s\in S}{\dprod }\left( Y^{X_{s}},\tau _{s_{p}}^{^{\prime
}},E\right) \rightarrow \left( \left( \underset{s\in S}{\dprod }Y_{s}\right)
^{X},\left( \underset{s\in S}{\dprod }\tau _{s}^{^{\prime }}\right)
_{p},E\right)
\end{equation*}

is a soft homeomorphism in the pointwise soft topology.
\end{theorem}

\begin{proof}
Since $\Delta $ is bijective mapping, to prove the theorem it is sufficient
to show that the mappings $\Delta $ and $\Delta ^{-1}$ are soft open. First,
we show that the mapping $\Delta $ is soft open. Let us take an arbitrary
soft set

\begin{equation*}
\left( e_{x_{\alpha _{1}}}^{-1}\right) _{s_{1}}(F_{s_{1}},E)\times ...\times
\left( e_{x_{\alpha _{k}}}^{-1}\right) _{s_{k}}(F_{s_{k}},E)\times \left( 
\underset{s\neq s_{1}...s_{k}}{\dprod }Y_{s}\right) ^{X}
\end{equation*}

belongs to the base of the product space $\underset{s\in S}{\dprod }\left(
Y^{X_{s}},\tau _{s_{p}}^{^{\prime }},E\right) .$ Since the soft set

\begin{eqnarray*}
&&\Delta \left( \left( e_{x_{\alpha _{1}}}^{-1}\right)
_{s_{1}}(F_{s_{1}},E)\times ...\times \left( e_{x_{\alpha _{k}}}^{-1}\right)
_{s_{k}}(F_{s_{k}},E)\times \left( \underset{s\neq s_{1}...s_{k}}{\dprod }%
Y_{s}\right) ^{X}\right) \\
&=&\{\left\{ f_{s}\right\} \mid f_{s_{1}}(x_{\alpha _{1}}^{1})\in
(F_{s_{1}},E),...,f_{s_{k}}(x_{\alpha _{k}}^{k})\in (F_{s_{k}},E)\} \\
&=&(F_{s_{1}}^{x_{\alpha _{1}}^{1}},E)\times ...\times (F_{s_{k}}^{x_{\alpha
_{k}}^{k}},E)\times \left( \underset{s\neq s_{1}...s_{k}}{\dprod }%
Y_{s}\right) ^{X}
\end{eqnarray*}

is soft open, $\Delta $ is a soft open mapping.

Similarly, it can be proven that $\Delta ^{-1}$ is soft open mapping.
Indeed, for each soft open set $e_{x_{\alpha }}^{-1}\left(
(F_{s_{1}},E)\times ...\times (F_{s_{k}},E)\times \underset{s\neq
s_{1}...s_{k}}{\dprod }Y_{s}\right) \in \left( \left( \underset{s\in S}{%
\dprod }Y_{s}\right) ^{X},\left( \underset{s\in S}{\dprod }\tau
_{s}^{^{\prime }}\right) _{p},E\right) ,$

\begin{eqnarray*}
&&\Delta ^{-1}\left( e_{x_{\alpha }}^{-1}\left( (F_{s_{1}},E)\times
...\times (F_{s_{k}},E)\times \underset{s\neq s_{1}...s_{k}}{\dprod }%
Y_{s}\right) \right) \\
&=&\Delta ^{-1}\left( \left\{ f:X\rightarrow \underset{s\in S}{\dprod }%
Y_{s}\mid f(x_{\alpha })\in (F_{s_{1}},E)\times ...\times
(F_{s_{k}},E)\times \underset{s\neq s_{1}...s_{k}}{\dprod }Y_{s}\right\}
\right) \\
&=&\left\{ p_{s}\circ f\mid f(x_{\alpha })\in (F_{s_{1}},E)\times ...\times
(F_{s_{k}},E)\times \underset{s\neq s_{1}...s_{k}}{\dprod }Y_{s}\right\} \\
&=&\left\{ p_{s}\circ f\mid p_{s_{1}}\circ f(x_{\alpha })\in
(F_{s_{1}},E),...,p_{s_{k}}\circ f(x_{\alpha })\in (F_{s_{k}},E)\right\} .
\end{eqnarray*}

Hence this set is soft open and the theorem is proved.
\end{proof}

Now, let $\left( X,\tau ,E\right) $, $\left( Y,\tau ^{^{\prime }},E\right) $
and $\left( Z,\tau ^{^{\prime \prime }},E\right) $ be soft topological
spaces and $f:\left( Z,\tau ^{^{\prime \prime }},E\right) \times \left(
X,\tau ,E\right) \rightarrow \left( Y,\tau ^{^{\prime }},E\right) $ be a
soft mapping. Then the induced map $\overset{\wedge }{f}:X\rightarrow Y^{Z%
\text{ }}$is defined by $\overset{\wedge }{f}(x_{\alpha })\left( z_{\beta
}\right) =f\left( x_{\alpha },z_{\beta }\right) $ for soft points $x_{\alpha
}\in \left( X,\tau ,E\right) $ and $z_{\beta }\in \left( Z,\tau ^{^{\prime
\prime }},E\right) $. We define exponential law

\begin{equation*}
E:Y^{Z\times X}\rightarrow \left( Y^{Z}\right) ^{X}
\end{equation*}

by using induced maps $E(f)=$ $\overset{\wedge }{f}$ \ i.e., $E(f)(x_{\alpha
})\left( z_{\beta }\right) =f\left( z_{\beta },x_{\alpha }\right) =\overset{%
\wedge }{f}(x_{\alpha })\left( z_{\beta }\right) .$ We define the following
mapping

\begin{equation*}
E^{-1}:\left( Y^{Z}\right) ^{X}\rightarrow Y^{Z\times X}
\end{equation*}

which is an inverse mapping $E$ as fallows

\begin{equation*}
E^{-1}(\overset{\wedge }{f})=f,\text{ \ }E^{-1}(\overset{\wedge }{f})\left(
z_{\beta },x_{\alpha }\right) =\text{ \ }E^{-1}(\overset{\wedge }{f}%
(x_{\alpha })\left( z_{\beta }\right) )=f\left( z_{\beta },x_{\alpha
}\right) .
\end{equation*}

Generally, \i n the pointwise topology for each soft continuous map $g$, the
mapping $E^{-1}\left( g\right) $ need not to be soft continuous. Let us give
the solution of this problem under the some conditions.

\begin{theorem}
Let $\left( X,\tau ,E\right) $, $\left( Y,\tau ^{^{\prime }},E\right) $ and $%
\left( Z,\tau ^{^{\prime \prime }},E\right) $ be soft topological spaces and
the mapping $e:Y^{Z}\times X\rightarrow Z,$ $e(f,z)=f(z)$ be soft
continuous. Function space $Y^{X}$ with pointwise soft topology and for each
soft continuous mapping $\overset{\wedge }{g}:X\rightarrow Y^{Z},$

\begin{equation*}
E^{-1}\left( \overset{\wedge }{g}\right) :Z\times X\rightarrow Y
\end{equation*}

is also soft continuous.
\end{theorem}

\begin{proof}
By using the mapping

\begin{equation*}
1_{Z}\times \overset{\wedge }{g}:Z\times Y\rightarrow Z\times Y^{Z},
\end{equation*}

we take

\begin{equation*}
Z\times X\overset{1_{Z}\times \overset{\wedge }{g}}{\longrightarrow }Z\times
Y^{Z}\overset{t}{\longrightarrow }Y^{Z}\times Z\overset{e}{\longrightarrow }%
Y.
\end{equation*}

Hence $e\circ t\circ \left( 1_{Z}\times \overset{\wedge }{g}\right) \in
Y^{Z\times X}$, where $t$ denotes switching mapping. Let us apply
exponential law $E$ to $e\circ t\circ \left( 1_{Z}\times \overset{\wedge }{g}%
\right) $. For each soft point $x_{\alpha }\in \left( X,\tau ,E\right) $ and 
$z_{\beta }\in \left( Z,\tau ^{^{\prime \prime }},E\right) $,

\begin{eqnarray*}
\left\{ \left[ E\left( e\circ t\circ \left( 1_{Z}\times \overset{\wedge }{g}%
\right) \right) \right] \left( x_{\alpha }\right) \right\} \left( z_{\beta
}\right) &=&\left( e\circ t\circ \left( 1_{Z}\times \overset{\wedge }{g}%
\right) \right) \left( z_{\beta },x_{\alpha }\right) \\
&=&e\circ t\left( z_{\beta },\overset{\wedge }{g}(x_{\alpha })\right) \\
&=&e\left( \overset{\wedge }{g}(x_{\alpha }),z_{\beta }\right) \\
&=&\left( \overset{\wedge }{g}(x_{\alpha })\right) \left( z_{\beta }\right) .
\end{eqnarray*}

Since $E\left( e\circ t\circ \left( 1_{Z}\times \overset{\wedge }{g}\right)
\right) =\overset{\wedge }{g},$ $E^{-1}\left( \overset{\wedge }{g}\right)
=e\circ t\circ \left( 1_{Z}\times \overset{\wedge }{g}\right) .$ Hence
evaluation maps $e$ and $t$ are soft continuous, $E^{-1}\left( \overset{%
\wedge }{g}\right) $ is soft continuous.
\end{proof}

\section{\textbf{\ CONCLUSION}}

\noindent We hope that, the results of this study may help in the
investigation of soft normed spaces and in many researches.

\end{document}